\documentclass[12pt,reqno]{amsart}
\RequirePackage[OT1]{fontenc}
\RequirePackage{amsthm,amsmath,bbm, array, graphicx, tikz, tikz-cd}
\RequirePackage[numbers]{natbib}
\RequirePackage[colorlinks,linkcolor=blue,citecolor=blue,urlcolor=blue]{hyperref}
\usepackage{natbib}
\usepackage{atbegshi}
\usepackage{amssymb}
\usepackage{anysize}
\usepackage{hyperref}
\usepackage{extarrows}
\usepackage{multirow}
\usepackage{natbib}
\usepackage{stmaryrd}
\usepackage{color}
\usepackage{amsmath}
\usepackage{enumitem}
\usepackage{mathtools} 
\usepackage{dsfont} 
\usepackage{verbatim}
\usepackage{amsfonts}
\usepackage{amsmath}
\usepackage{amssymb}
\usepackage{amscd}
\usepackage{amsthm}
\usepackage{latexsym,bm}
\usepackage{pstricks}
\usepackage{geometry}
\usepackage{dsfont}
\usepackage{bigints}
\usepackage{bm}
\usepackage{soul}
\usepackage{amsfonts}
\makeatletter
\def\amsbb{\use@mathgroup \M@U \symAMSb}
\makeatother

\usepackage{bbold}

\numberwithin{equation}{section}
\theoremstyle{plain} 
\newtheorem{theorem}{Theorem}[section]
\newtheorem{lemma}[theorem]{Lemma}
\newtheorem{corollary}[theorem]{Corollary}
\newtheorem{proposition}[theorem]{Proposition}

\theoremstyle{definition}
\newtheorem{example}[theorem]{Example}
\theoremstyle{remark}
\newtheorem{remark}[theorem]{Remark}

\newtheorem{notation}[theorem]{Notation}
\newtheorem{question}[theorem]{Question}

\newcommand{\catvec}{\textbf{Vec}}
\renewcommand{\P}{\mathcal{P}}

\newcommand{\Hom}{\text{Hom}}
\renewcommand{\Vec}{\textbf{Vec}}
\newcommand{\1}{\textbf{1}}
\newcommand{\bbC}{\mathbb{C}}
\newcommand{\bbQ}{\mathbb{Q}}

\def\be{\begin{equation}}
\def\ee{\end{equation}}
\def\ba{\begin{eqnarray*}}
\def\ea{\end{eqnarray*}}
\def\bae{\begin{eqnarray}}
\def\eae{\end{eqnarray}}
\def\bc{\begin{center}}
\def\ec{\end{center}}

\def\D{{\bm D}}

\def\sl{\mathfrak{sl}}

\def\Fus{\mathrm{Fus}}
\def\C{\mathcal{C}}
\def\D{\mathcal{D}}

\def\Z{\mathbb{Z}}

\def\Eq{\text{Eq}}

\def\<{\langle}
\def\>{\rangle}

\def\ea{e_{0}(\lambda)}

\allowdisplaybreaks
  
\marginsize{28mm}{28mm}{28mm}{29mm}  

\begin{document}

\title{Reconstructing Braided Subcategories of $SU(N)_k$}
\author{Zhaobidan Feng}
\author{Shuang Ming}
\author{Eric C. Rowell}
\maketitle
\begin{abstract}
    Ocneanu rigidity implies that there are finitely many (braided) fusion categories with a given set of fusion rules. While there is no method for determining all such categories up to equivalence, there are a few cases for which can.  For example, Kazhdan and Wenzl described all fusion categories with fusion rules isomorphic to those of $SU(N)_k$.  In this paper we extend their results to a statement about braided fusion categories, and obtain similar results for certain subcategories of $SU(N)_k$.  
\end{abstract}
\section{Introduction} 
The purpose of this article is to study certain subtleties on the problem of classifying braided fusion categories with a fixed set of fusion rules.  Some of our results are possibly well-known to experts, but have not been carefully written down.  We assume the reader is familiar with the basic notions of the theory of fusion categories, taking \cite{EGNO} as a basic reference.

In explicit classifications of braided fusion categories (eg. \cite{rowellstongwang,BNRW}) one is often confronted with the following question: if $\C$ and $\D$ have the same fusion rules (i.e. are \emph{Grothendieck equivalent} \cite{RowellJPAA}), are they related in some explicit way?  By (braided) Ocneanu rigidity \cite[Theorem 2.31]{eno} there are finitely many (braided) fusion categories with the same fusion rules as a given one, but this does not provide a classification or even an enumerative bound up to equivalence.  Often it is desirable to have such an enumeration, for example in categories appearing in applications to condensed matter physics and quantum computation \cite{rowellwangbulletin}.

There are two straightforward ways to construct potentially inequivalent fusion categories from a given category $\C$.  

Firstly, it is a consequence of results of \cite{eno} that any 
fusion category over $\bbC$ can be defined over an algebraic 
extension $K$ of $\bbQ$.  In \cite{davidovitchhaggewang} it is carefully shown that the axioms of a (braided, ribbon) fusion category can be expressed as structure constants satisfying algebraic equations so that for any Galois automorphism $\sigma$ one may define a (braided, ribbon) fusion category $\sigma(\C)$ by applying $\sigma$ to the structure constants.  Since the fusion coefficients are rational integers, the fusion rules of $\C$ and $\sigma(\C)$ are the same.  

If $\C$ is a faithfully $G$-graded fusion category with associativity constraint $\alpha$ then for any $3$-cocycle $\omega\in Z^3(G,U(1))$ we may obtain a new fusion category $\C^\omega$ by twisting $\alpha$ by $\omega$ on homogeneous components:
$$\alpha_{X,Y,Z}^\omega=\omega(\deg(X),\deg(Y),\deg(Z))\alpha_{X,Y,Z}.$$ Indeed, the pentagon axioms correspond exactly to the cocycle condition and twisting by cohomologous 3-cocycles yield monoidally equivalent categories.  

In some situations these two constructions suffice to describe all categories with given fusion rules.  For example, any fusion category with fusion rules like $Vec_G$ for a finite group $G$ is of the form $Vec_G^\omega$.  
The results of Kazhdan and Wenzl \cite{kazdanwenzl} show that the same is true for fusion categories with the same fusion rules as the $\Z_N$-graded fusion categories $SU(N)_k$ obtained from quantum groups $U_q\mathfrak{sl}_N$ for $q=e^{\pi i/(N+k)}$.  They show that if $\C$ has fusion rules like $SU(N)_k$ then $\C$ is a twist of the fusion category $\Fus(\C(\sl_N,N+k,\tilde{q}))$ obtained from $U_{\tilde{q}}\mathfrak{sl}_N$ where $\tilde{q}$ is another root of unity of the same order as $q$. The results mentioned above make it clear that $q\rightarrow \tilde{q}$ can be implemented by a Galois automorphism. Their approach is fairly technical and uses the relationship between the Hecke algebras $\mathcal{H}_n(q)$ and the centralizer algebras in $SU(N)_k$ in an essential way. 

The categories $SU(N)_k$ admit a further structure and properties: they are non-degenerate braided fusion categories.
Moreover $SU(N)_k$ has a well-studied factorization into braided subcategories $\mathcal{M}SU(N)_k\boxtimes\C(\Z_m,P)$: here $\C(\Z_m,P)$ is a pointed modular category with fusion rules like $\Z_m$ where $N/m$ is the largest factor of $N$ coprime to $k$, and $\mathcal{M}SU(N)_k$ is the centralizer of $\C(\Z_m,P)$.  In the case $\gcd(N,k)=1$ the category $\mathcal{M}SU(N)_k$ is often denoted $PSU(N)_k$, see e.g. \cite{bruguieres}.

The motivating question for this paper is:
\begin{question}  Can we classify braided fusion categories with the same fusion rules as $SU(N)_k$ or $\mathcal{M}SU(N)_k$ up to braided equivalence?
\end{question}
We shall be particularly interested in describing all \emph{non-degenerate} braidings on such categories.

For a fixed braided fusion category  $\C$ one can also change the braiding in a number of ways.  Firstly, one may always reverse the braiding to obtain a (potentially) new braided fusion category $\C^{rev}$ with the same underlying fusion category: the braiding on $\C^{rev}$ is defined to be $\tilde{c}_{X,Y}:=(c_{Y,X})^{-1}$ where the braiding on $\C$ is given by $c$.  If $\C$ is a braided faithfully $G$-graded fusion category and $\chi:G\times G\rightarrow U(1)$ is a bicharacter then we we can equip the underlying fusion category $\Fus(\C)$ with a (potentially) new braiding by defining on homogeneous objects $X,Y$
$$c^\chi_{X,Y}=\chi(\deg(X),\deg(Y))c_{X,Y}.$$  The proof that this is valid is essentially by inspecting the hexagon equations, and goes back to Joyal and Street at least in some cases \cite{JS}.   
Finally, for any tensor autoequivalence $\phi$ of $\C$, the image of the braiding on $\C$ under $\phi$ is a braiding on the underlying fusion category of $\C$.  In particular if $\phi$ is not braided (recall \cite{EGNO} that being braided is a property autoequivalences may or may not have)  
then we obtain a (potentially new) braiding on the fusion category $\Fus(\C)$.

It is natural to ponder the possibility of first twisting the associativity on a $G$-graded fusion category $\C$ and then changing the braiding correspondingly. This leads to the notion of \emph{abelian $3$-cocycles} $(\omega,\chi)$ (see \cite[Exercise 8.4.3]{EGNO}), where $\omega$ is a $3$-cocycle as above, and $\chi:G\times G\rightarrow U(1)$ is a function (not necessarily a bicharacter, unless $\omega$ is trivial).  The pentagon and hexagon axioms constrain $\omega$ and $\chi$ significantly--for example, if $|G|$ is odd then the only abelian $3$-cocycles have $\omega$ trivial.  This can be regarded as a special case of (braided) \emph{zesting} introduced in \cite{zesting}.

For completeness, we mention that there is an additional structure on $SU(N)_k$: they are spherical (non-degenerate) braided fusion categories with canonical spherical structure coming from the standard ribbon twist.  This structure may also be changed--for a non-degenerate braided fusion category that admits a spherical structure other spherical structures are in one-to-one correspondence with self-dual invertible objects \cite{BNRW}.  The same is true for the non-degenerate subcategories $\mathcal{M}SU(N)_k$.

We fully answer the question above.  In particular, we describe all \emph{braided} fusion categories with the fusion rules like those of the modular categories $SU(N)_k$ and of the modular subcategories $\mathcal{M}SU(N)_k$.  

The structure of this paper is as follows. In section 2, we lay out the basic definitions and general results about braided tensor categories. We also fix the notations that will be used in the sequel sections. In section 3, we introduce basic properties of the categories $\C(\sl_N,N+k,q)$. In section 4, we classify all possible braidings over $\C(\sl_N,N+k,q)$. In section 5, we classify all possible braidings over categories obtained from $\C(\sl_N,N+k,q)$ by twisting the associativity constraints, we also classify all possible braidings over fusion categories with the same fusion rules as certain subcategories of $SU(N)_k$.

\section{Preliminaries}
In this section, we fix notations coming from the general theory of braided fusion categories. We refer readers to \cite{EGNO} for more details.

A fusion category over $\mathbb{C}$ is a $\mathbb{C}$-linear semisimple rigid monoidal category with finitely many isomorphism classes of simple objects, finite-dimensional Hom-spaces. We denote $\mathcal{O}(\C)$ be the set of all isomorphism classes of simple objects of $\C$. An object $X$ is said pointed if the evaluation morphism $X^{*}\otimes X\rightarrow \1$ is an isomorphism. A fusion category is said to be pointed if all simple objects are pointed. Given a fusion category $\C$, taking the fusion subcategory that generated by pointed objects form a pointed fusion category.

We denote the pointed fusion category of all $\bbC$-vector spaces by $\Vec$.

Let $G$ be a finite group. A monoidal category $\C$ is $G$-graded if $\C\cong \bigoplus_{g\in G}\C^g$ as abelian categories and $\C^g\otimes \C^h\subset \C^{gh}$. In this case there is a function $\deg:\mathcal{O}(\C)\rightarrow G$ given by $\deg(X)=g$ if $X\in\C^g$. In particular, if an object $Z$ is a subobject of the tensor product of simple objects $X\otimes Y$ then $\text{deg}(Z)=\text{deg}(X)\text{deg}(Y)$. 
We say the grading is faithful if deg is surjective.  Notice that the grading only depends on the fusion rules.

\begin{example} A pointed fusion category, i.e., one in which all simple objects are invertible under $\otimes$ are automatically $G$-graded, where $G$ is the group of isomorphism classes of simple objects with product induced by $\otimes$.
The category of $G$-graded vector spaces is a pointed fusion category with $G$ grading. We denote it by $\Vec_{G}$.  The related pointed fusion category $\Vec_G^\omega$ is obtained by twisting the associativity on $\Vec_G$ by a 3-cocycle $\omega$.  Pointed braided fusion categories are classified by pre-metric groups and the underlying group is abelian, see \cite{EGNO}.
\end{example}

Let $\C$ be a braided fusion category with braiding $c_{X, Y}$. We say two objects $X$ and $Y$ centralize each other if $c_{Y, X}\circ c_{X, Y}=\textbf{id}_{X\otimes Y}$, and projectively centralize each other if 
$$c_{Y, X}\circ c_{X, Y}=b_{X, Y}\textbf{id}_{X\otimes Y}$$ for some scalar $b_{X,Y}$.
In particular, pointed objects always projectively centralize simple objects.

\section{Universal Grading and Decomposition of $SU(N)_k$}
We briefly describe some of the relevant notation for the categories $Rep(SL(N))$ of complex $SL(N)$ representations and $SU(N)_k$ the modular fusion category associated with quantum groups of Lie type $A_{N-1}$ the specific root of unity $q=e^{\pi i/(N+k)}$.  For more complete details we refer to \cite{bakalovkirillov,kazdanwenzl,Wenzl1988}.
\subsection{Combinatorial Data} 
The monoidal category of complex $SL(N)$-representations is semisimple: the isomorphism classes of simple objects are parametrized by the set $\Lambda_N$ of Young diagrams $\lambda$ with at most $N-1$ rows.  These are either written row-wise as $(m_1,\ldots,m_{N-1})$ where the weakly decreasing $m_i$ represent the number of boxes in the $i$th row, or as $[\lambda_1,\ldots,\lambda_j]$ where $\lambda_i$ represents the number of boxes in the $i$th column, with $\lambda_j$ the last non-empty column, unless $\lambda_1=0$. For instance $X_{[0]}=X_{(0,\ldots,0)}$ denotes the unit object corresponding to the trivial representation. The object labeled by a single box $X_{[1]}=X_{\Box}$ is the generating object which corresponds to the $N$-dimensional fundamental representation. In the rest of the paper, we will use $X$ to denote the generating object $X_{[1]}$ for simplicity.

The fusion rules of $Rep(SL(N))$ satisfy
$$X\otimes X_\mu\cong\bigoplus_{\lambda=\mu+\Box} X_\lambda,$$ where $\lambda=\mu+\Box$ indicates that $\lambda$ is obtained from $\mu$ by adding one box to any row/column of $\mu$, with the convention that if $\mu$ has $N-1$ rows, then instead of adding one box to the first column one deletes the first column. For example the object $X_{[1^{N-1}]}$ labeled by a column of $N-1$ boxes is the dual object to $X$, since $X_{[0]}\subset X\otimes X_{[1^{N-1}]}$. The Grothendieck semi-ring of this fusion category is a based $\Z_+$-ring with basis parametrized by Young diagrams with at most $N-1$ rows, and the product is the obvious one coming from the tensor product.

\subsection{Fusion categories} The braided fusion categories $SU(N)_k$ are obtained as a subquotients of the categories $Rep(U_q\mathfrak{sl}_N)$ with $q=e^{\pi i/(N+k)}$, see e.g., \cite{bakalovkirillov} for details. The fusion rules of $SU(N)_k$ are truncated versions of those presented above for representations of $SL(N)$. To be precise, the fusion rules $SU(N)_k$ can be described as follows: we restrict to the objects with at most $k$ columns. 

\subsection{Universal Grading} There is an universal $\Z_N$ grading on both the representation categories of $SL(N)$ and the associated fusion categories $SU(N)_k$ by counting the number of boxes of Young diagrams mod $N$. For instance, the generating object $X$ is of grade $1$ and the trivial object $X_{[0]}$ is of grade $0$. 

There are (braided) fusion subcategories coming from  the universal grading. Suppose $H$ is a subgroup of $\Z_N$, then $\bigoplus_{h\in H} (SU(N)_k)^h$ is a (braided) fusion subcategory.  
\subsection{Pointed Subcategory} There are exactly $N$ invertible objects in $SU(N)_k$. The corresponding Young diagrams are those of rectangular shape $i\times k$, where $0\le i\le N-1$. If $\C$ is a braided fusion category with the same fusion rules as $SU(N)_k$  the braided monoidal structure restricts to the pointed subcategory $\P(\C)$ of $\C$. The braiding and monoidal structures of pointed categories are completely classified, see \cite{EGNO} Section 8.4. In our cases, in order to be braided, the pointed subcategory $\P(\C)$ must be monoidally equivalent to $\catvec_{\Z_N}^{\omega}$, where
\begin{enumerate}
    \item for $N$ even either $$\omega(i, j, \ell)=\begin{cases} 1 & i+j<N\\
(-1)^{\ell} & i+j\geq N \end{cases}$$
or [$\omega$] is homologically trivial, and
\item $[\omega]$ is homologically trivial if $N$ is odd.
\end{enumerate}
We put the concrete computation in the appendix. In $SU(N)_k$. We denote the pointed simple object $X_{[k]}$ by $g$. 

\subsection{Decomposition Formula} One can derive fusion subcategories of $SU(N)_k$ from both the grading and the pointed objects. However, not all of them split as a direct (Deligne) product of braided fusion categories, or even as fusion categories.  
The following proposition shows that any braided fusion category $\C$ with the same fusion rules as $SU(N)_k$ does have such a decomposition, which is maximal in a certain sense.  For such a $\C$, let $m$ be the largest divisor of $N$ that is relatively prime to $k$, and set $n=N/m$. In some literature, they define $n=\text{gcd}(N, k^{\infty}):=\lim_{i\rightarrow \infty}\text{gcd}(N, k^{i})$.  Furthermore, denote by $\mathcal{M}\C=\bigoplus_{i=0}^{n-1}\C^{im}$ the fusion subcategory of $\C$  generated by the $mj$-graded components (i.e. corresponding to the subgroup $m\Z_N<\Z_N$), and by $\C(\Z_{m},P)$ the pointed subcategory generated by $g^n$. $\C(\Z_m,P)$  has rank $m$ since $ni< N$ for $i<m$. Notice that since $g$ lies in the ${k}$-graded component, $g^n$ lies in the $nk\pmod{N}$ component. The intersection of the two fusion subcategories is trivial since $n$ and $m$ are relatively prime.

\begin{proposition}\label{decom}
 A braided monoidal category $\C$ with the same fusion rules as $SU(N)_{k}$ admits a braided tensor decomposition (in the notation above):
\begin{equation}
\label{decomp}
 \C=\mathcal{M}\C\boxtimes\C(\Z_m,P).
\end{equation}
\end{proposition}

\begin{proof}

It is clear by the construction above that $\mathcal{M}\C$ and $\C(\Z_m,P)$ are fusion subcategories.
Since $\C$ is braided, one has a monoidal functor 
$$F:\mathcal{M}\C\boxtimes\C(\Z_{m},P)
\rightarrow \C$$
given by the tensor product.
We will verify $F$ is an equivalence of fusion categories by showing that $F$ induces a bijection on the set of simple objects. By the Chinese remainder theorem, there exists a group isomorphism $f: \Z_{n}\times \Z_{m}\rightarrow \Z_{N}$, this proves the injection since the simple objects in $\C(\Z_m,P)$ are all invertible. For the surjection, let $V$ be a simple object in $\C$ of grade $\ell$, and $f(i,j)=\ell$. In another words, $$im+jn=\ell\pmod{N}.$$ Then 
$V=(V\otimes g^{-jn})\otimes g^{jn}$
where $V\otimes g^{-jn}$ is of grade $in$, thus an object is $\mathcal{M}\C$, and $g^{jn}$ is an object in $\C(\Z_m,P)$.

To show $F$ is an equivalence of \emph{braided} tensor categories, we need to show the pointed factor $\C(\Z_m,P)$ centralizes  $\mathcal{M}\C$.

Let $X$ and $g$ be the simple objects defined above. Since the pointed objects centralize all simple objects projectively, the quantity $b_{Y,g^i}$ characterizes the braiding of a pointed object $g^i$ and a simple object $Y$, where $b_{Y,g^i}$ is defined via
$$c_{g^i, Y}\circ c_{Y, g^i} =b_{Y, g^i}\textbf{id}_{Y\otimes g^{i}}.$$
By the functoriality of braiding, the pointed objects projectively centralize $X^{\otimes j}$ for all $j$, with $b_{X^{{\otimes}j}, g^i}=b_{X, g}^{ij}$. Since $X$ is a generating object. The quantity $b_{Y, g^i}$ is determined by $b_{X, g}$ for all simple object $Y$. To be specific, suppose $b_{X, g}=t$ and $Y\in \C^j$, then $b_{Y, g^i}=t^{ij}$.

Since the identity object $\1$ centralizes all object in $\C$, we have
$$b_{X, \1}=b_{X, g^N}=t^N=1.$$

Now we prove $\C(\Z_m,P)$ centralizes all objects in $\mathcal{M}\C$. The pointed subcategory $\C(\Z_m,P)$ is generated by $g^{n}$. Let $Y$ be a simple object in $\C^{is}$. Then
$$b_{Y, g^{n}}=t^{(im)(n)}=t^{Ni}=1.$$
Thus the generating object $g^{n}$ centralizes all objects in $M\C$, therefore the same holds all other objects in $\C(\Z_m,P)$.
\end{proof}

With the notation established above we immediately have the following:
\begin{corollary}\label{unique}
The braiding over $\C$ is uniquely determined by a braiding over $\mathcal{M}\C$ and a braiding over $\C(\Z_m,P)$.
\end{corollary}

\begin{remark}
The pointed factor $\C(\Z_m,P)$ is maximal in the sense that one can not find a braided tensor decomposition of $\C$ such that the pointed factor is larger than $\C(\Z_m,P)$. 
\end{remark}

\begin{notation}
In the rest of the paper, we will denote by $\mathcal{M}\C$ the non-pointed factor in the decomposition of a braided fusion category with the same fusion rules as $SU(N)_k$ and  $\C(\Z_m,P)$ to denote the (maximal) pointed factor for convenience.
\end{notation}

\subsection{Autoequivalences of $SU(N)_k$}
The (braided) monoidal autoequivalences of the category $SU(N)_k$ are classified by  Edie-Michell \cite{ediesimple,edie2020auto} and  Gannon \cite{gannon2002automorphisms}.
Gannon first classified all the automorphisms of the fusion rings of $SU(N)_k$. They are generated by two types of automorphisms, namely:

\begin{enumerate}
\item \textit{charge conjugation} that interchanges the classes $[X]$ and $[X^*]$
\item \textit{simple current automorphisms} that sends the generating object class $[X]$ to $[X\otimes g^{a}]$ for any $a$ such that $1+ka$ is coprime to $N$.
\end{enumerate}

Then Edie-Michell \cite{ediesimple} showed that all such fusion ring autoequivalences can be realized  uniquely as a monoidal equivalence of $SU(N)_k$, and determined when they are braided. In particular, charge conjugation always induces a {braided} monoidal equivalence of $SU(N)_k$. The simple current autoequivalences may change the braiding or not. In particular, if we apply an autoequivalence that is not braided we obtain a new braiding on our category.  The groups of (braided) simple current autoequivalences is  given in Table \ref{EMtheorem}, in which $m$ is the largest factor of $N$ coprime to $k$ and $n=N/m$, $p$ is the number of distinct odd primes that divide $N$ but not $k$, and 
$$t=\begin{cases} 0, & N \text{ is odd,}\\ 0 & N \text{ is even and } k \equiv 0 (\text{mod } 4), \text{ or if } k \text{ is odd and } N \equiv 2 (\text{mod } 4),\\
1, & \text{otherwise}. \end{cases}$$
\begin{table}
\begin{center}
    \begin{tabular}{|c|c|c|}
    \hline
    $SU(N)_k$ &ScEq& BrScEq\\
    \hline
    $k=2, N=2$&$\{e\}$&$\{e\}$\\
        $2$ exactly divides $\gcd(N,k)$ &$\Z_{m}^{*}\times \Z_2\times \Z_{\frac{n}{2}}$ & $\Z_{2}^{p+t}$\\
        otherwise   & $\Z_{m}^{*}\times \Z_{n} $&$\Z_{2}^{p+t}$\\
        \hline
    \end{tabular}
\end{center}\caption{Simple Current Autoequivalences of $SU(N)_k$\label{EMtheorem}}
\end{table}

Part of the results in \cite{edie2020auto} applies to any category $\C$ with the same fusion rules as $SU(N)_k$ provided $\C$ admits a braiding. That is, any simple current automorphism of the fusion ring lifts to an autoequivalence of $\C$.  We do not know if it lifts uniquely: the issue is that we do not know the trivial simple current automorphism is only realized by the identity autoequivalence: there could be non-trivial gauge autoequivalences. On the other hand, we show that all such autoequivalences are braided, see \ref{gauging}, they can be ignored from the perspective of counting distinct braidings.  To see that the results are valid for such a $\C$, we show that the fusion ring automorphism that sends $[X]$ to $[X\otimes g^{a}]$ lifts to an autoequivalences.

Suppose $\C$ admits a braiding, then the braiding equips $\otimes: \C\boxtimes \C\rightarrow \C$ with the structure of a monoidal functor. Then one can restrict the monoidal functor to a subcategory $\C\boxtimes \langle g^{a}\rangle$, the autoequivalence $f$ can be constructed as the following commuative diagram.

$$\begin{tikzcd}
\C\boxtimes \langle g^a\rangle\arrow[r, "\text{Forget}"] &\C\boxtimes \Vec\arrow[d, "p_1"] \\
\C\arrow[u, "s"]\arrow[r, "f"] & \C 
\end{tikzcd}$$

where $s$ is a section functor that sends $X$ to $X\boxtimes g^{-a}$. It is easy to verify that the associator over $\langle g^a\rangle$ is trivial if $[X]\rightarrow [X\otimes g^{a}]$ defines a fusion ring automorphism. which makes the forgetful functor monoidal.

\section{Main Results}

In this section, we let $\Fus(\C(\sl_N,\ell, q))$ denote the monoidal category underlying $\C(\sl_N,\ell, q)$ equipped with the standard (untwisted) associativity constraints. We only discuss the cases with $k\ge 2$, so the category is not pointed. Since $\Fus(\C(\sl_N,\ell, q))$ admits a non-degenerate braiding (Galois conjugation does not change the degeneracy of the $S$-matrix), the results of \cite{EGNO} and \cite{NIKSHYCH18} on classifying braidings over fusion categories can be applied.

Recall that if $\mathcal{C}$ is a non-degenerate braided fusion category then the Drinfeld center of $\mathcal{C}$ is braided equivalent to $\mathcal{C}\boxtimes\mathcal{C}^{rev}$,  where $\mathcal{C}^{rev}$ denotes the braided fusion category $\mathcal{C}$ equipped with the reversed braiding $c_{X, Y}^{rev}:=c_{Y, X}^{-1}$. In particular, the braidings of $\mathcal{C}$ are in one-to-one correspondence with the sections of 
\begin{align*}
\mathcal{C}\boxtimes\mathcal{C}^{rev}&\rightarrow \mathcal{C} \\
X\boxtimes Y &\mapsto X\otimes Y,
\end{align*}
where the associated braiding is defined as
$$\tilde{c}_{X_1\boxtimes Y_1, X_2\boxtimes Y_2}=c_{X_1, X_2}\boxtimes c_{Y_2, Y_1}^{-1}.$$

With the above identification, we can classify all braidings over $\Fus(SU(N)_k)$ when $k\ge 2$. We remark that the classification of braidings over $\Fus(SU(N)_1)$ is well-known as they are all pointed categories with the same fusion rules as $\Vec_{\Z_{N}}$, see \cite{EGNO} and the Appendix for details.

\begin{theorem}
\label{untwist}
For $k\ge 2$, there are exactly $2N$ different braid structures over fusion category $\Fus(\C(\sl_N,k+N, q))$. In particular, $\Fus(\C(\sl_N,k+N,q))$ admits a degenerate braiding if and only if $N$ has an odd prime factor which is relatively prime to $k$.
\end{theorem}
\begin{proof}
Let $\C=\Fus(\C(\sl_N,k+N,q))$ equipped with the nondegenerate braiding coming from the Galois conjugation of $SU(N)_{k}$. It suffices to classify all sections of $\C$ in its center $\C\boxtimes \C^{rev}$.

Suppose $s:\C\rightarrow \C\boxtimes \C^{rev}$ is a section. Since $\C$ is tensor generated by $X$, the section $s$ is uniquely determined by its image $s(X)$ in $SU(N)_{k}\boxtimes SU(N)_{k}^{rev}$. Notice that $X$ is simple, one of the factors must be pointed. Thus, $s(X)$ is either of the form $X\otimes g^{i}\boxtimes g^{-i}$ or the opposite $g^{-i}\boxtimes X\otimes g^{i}$ where $0\le i\le N-1$. We hence obtain in total $2N$ choices of braiding. This finishes the proof of the first part of Theorem \ref{untwist}.

Next we check the symmetric center of the corresponding braiding. We only consider the case when $s(X)=X\otimes g^{i}\boxtimes g^{-i}$, the other case is identical. Since the symmetric center of the induced braiding remains pointed(See Corollary 4.5 of \cite{NIKSHYCH18}), it suffices to compute $\tilde{b}_{s(X), s(g)}$, where $\tilde{b}_{Y, h}$ is defined such that
$$\tilde{b}_{Y, h}\textbf{id}_{Y\otimes h}=\tilde{c}_{g, Y}\circ \tilde{c}_{Y, g}$$ in $\C\boxtimes \C^{rev}$.

Let $b_{X, g}=t$ in $\C$. Since the symmetric center of $\C$ is trivial, $t$ is a primitive $N$-th root of unity. Otherwise, there exists $g^{i}\ne \1$ that centralizes $X$, as a result, centralizes the whole category. Notice $g$ appears in $X^{\otimes k}$, then $s(g)=g\otimes g^{ik} \boxtimes g^{-ik}$. Thus we have
\begin{equation}\label{equat}
\begin{split}
\tilde{b}_{s(X), s(g)}&=\tilde{b}_{X\otimes g^i\boxtimes g^{-i}, g^{ik+1}\boxtimes g^{-ik}}\\
&=b_{X\otimes g^i, g^{ik+1}}b_{g^{-i}, g^{-ik}}^{-1}\\
&=b_{X, g^{ik+1}}b_{g^i, g^{ik+1}}b_{g^{-i}, g^{-ik}}^{-1}\\
&=t^{ik+1}t^{ik(ik+1)}t^{-(-ik)(-ik)}\\
&=t^{2ik+1}
\end{split}
\end{equation}
Since $t$ is a primitive $N$-th root of unity, $g^j$ is in the symmetric center if and only if 
\begin{equation}\label{e1}
(2ik+1)j\equiv 0 \mod N.     
\end{equation}

Now we are ready to prove the second part of Theorem \ref{untwist}. Notice that $2ik+1$ is relatively prime to $2$ and all the common prime factors of $N$ and $k$. This proves the 'only if' part. To prove the 'if' part, we construct a degenerate braiding over the underlying fusion category. Let $m$ be the maximal \emph{odd} divisor of $N$ that relatively prime to $k$ as in Section 3.5. According to Proposition \ref{decom},  $\C$ has a braided factor $\C(\Z_{m},P)$ with trivial associativity constraint. In particular this factor admits a symmetric braiding. By choosing this symmetric braiding over $\Fus(\C(\Z_{m},P))$ , we construct a degenerate braiding over $\C$. We hence finish the proof of Theorem \ref{untwist}.

\end{proof}






\section{Generalizations}
\label{twisted}
\subsection{Group Cohomology}
In \cite{kazdanwenzl}, Kazhdan and Wenzl classified all monoidal structures over categories with the same fusion rules as $SU(N)_k$. Different monoidal structures can be obtained by twisting the associators by (a cocycle representative of) a class in the third cohomology group of $\Z_N$ with coefficients in $U(1)$ and/or changing the choice of a primitive root of unity.

It is well-known that the monoidal structures over the category of $G$-graded vector spaces are in one to one correspondence with the elements in $H^{3}(G, U(1))$ as follows. We briefly recall the construction, see \cite{EGNO} for details.

Let $\Vec_{G}$ be the skeletal category with associativity constraints equal to the identity morphisms. Suppose $\omega$ is a $3$-cocycle representing $[\omega]\in H^3(G, U(1))$. We denote by $\Vec_{G}^{\omega}$ the category with the same fusion rules as $\Vec_{G}$, with associativity constraints replaced by $a_{g_1, g_2, g_3}=\omega(g_1, g_2, g_3)\textbf{id}$. It is easy to check that the pentagon axiom is equivalent to the condition that $\omega$ is a cocycle
$$\omega(g_1g_2, g_3, g_4)\omega(g_1, g_2, g_3g_4)=\omega(g_1, g_2, g_3)\omega(g_1, g_2g_3, g_4)\omega(g_2,g_3,g_4).$$
The category $\Vec_{G}^{\omega}$ is also called $G$-graded vector space twisted by $\omega$. The construction above applies to all $G$-graded categories in a straightforward way, one replace the associativity constraint $a_{X, Y, Z}$ by $\omega(g_X, g_Y, g_Z)a_{X, Y, Z}$ if $X, Y, Z$ belong to grade $g_X, g_Y, g_Z$ respectively. The theorem below by Kazhdan and Wenzl says that all monoidal categories with fusion rules the same as  $SU(N)_k$ are obtained from $\Fus(SU(N)_{k})$ by such cocycle twists and/or by changing the quantum parameter $q$.

\begin{theorem}[\cite{kazdanwenzl}]
Let $\mathcal{C}$ be a rigid monoidal semisimple $\mathbb{C}$-category such that the fusion ring is isomorphic to the fusion ring of $SU(N)_k$. Then $\mathcal{C}$ is monoidally equivalent to $\C(\sl_N,k+N,q)^{\omega}$ for some primitive root of unity $q$ of order { $2(k+N)$}  uniquely determined up to $q\rightarrow q^{-1}$, and some $3$-cocycle $\omega\in Z^3(\Z_N,U(1))$.
\end{theorem}

To characterize such twists, we give explicit representatives of elements in $H^{3}(\Z_N, U(1))$.
\begin{proposition}
Let $\mathbb{Z}_N=\Z/N\Z=\{0,1,\ldots,N-1\}$. The third cohomology group $H^{3}(\Z_N, U(1))\cong\Z_N$, with elements represented by cocycles
$$
\omega_{\eta}(i,j,k)=\begin{cases} 1 & i+j<N\\
\eta^{k} & i+j\geq N \end{cases}$$
where $\eta$ is an $N$-th root of unity.
\end{proposition}

Among the twisted fusion categories, the category admit a braiding only if $\eta=\pm 1$.(See the erratum of \cite{tuba2005braided}) In particular, $\eta$ takes the value $-1$ only if $N$ is an even number. In this section, we will classify all braidings over the underlying fusion category obtain by twisting the $\C(\sl_N,k+N,q)$ by $-1$.

Before we state our classification theorem, we first fix some notations. Let us denote the twisted category by $\C(\sl_N,k+N,q)^{-}$, and equip it with a braiding as follows. Without loss of generality, we assume $\C(\sl_N,k+N,q)$ is skeletal with associative constraint $a_{-, -, -}$ and braiding $c_{-, -}$. Then one can get a skeletal category that equivalent to $\C(\sl_N,k+N,q)^{-}$ by replacing $a_{W, Y, Z}$ by $a'_{W, Y, Z}=\omega_{-1}(g_W, g_Y, g_Z)a_{W, Y, Z}$. We further replace the braiding morphism by $c'_{Y, Z}=s^{g_Y+g_Z}c(Y, Z)$. It is easy to verify $a'$ and $c'$ satisfies the hexagon equations if $s^{N}=-1$ (See Appendix). Suppose $b_{X, g}=t$ in $\C(\sl_N,k+N,q)$ and $N=2^{p}r$ where $r$ is odd. We choose $s=\sqrt{t^{r}}$ for the braiding of $\C(\sl_N,k+N,q)^{-}$.


\subsection{Braidings over $Fus(\C(\sl_N,k+N,q)^{-})$}
\begin{theorem}
For $k\ge 2$,
\begin{enumerate}
\item If $k$ is odd, there are exactly $2N$ different braid structures over $Fus(\C(\sl_N,k+N,q)^{-})$, all of which are degenerate.
\item If $k$ is even, there are exactly $2N$ different braid structures over $Fus(\C(\sl_N,k+N,q)^{-})$, In particular, the category admit a non-degenerate braiding.
\end{enumerate}
\begin{proof} Let $\C=\C(\sl_N,k+N,q)$ as in $\ref{untwist}$ and $\C^{-}$ denote the braided fusion category $\C(\sl_N,k+N,q)^{-}$ constructed as above.
\begin{enumerate}
    \item [Case 1:] $k$ is odd. \\
    Since $\C$ and $\C^{-}$ are both braided, we write the decomposition of $\C$ and $\C^{-}$ as in Proposition \ref{decomp}
    \[\C=\mathcal{MC}\boxtimes \C(\Z_{m},P)\]
    \[\C^{-}=\mathcal{MC}^{-}\boxtimes \C(\Z_{m},P^{-}).\]
    Notice that the grading of objects in the $\mathcal{MC}^{-}$ are all multiples of $2^{p}$, thus $\mathcal{MC}^{-}$ is braided equivalent to $\mathcal{MC}$ and hence nondegenerate. On the other hand, the total number of braidings over $\C$ is $2N=2mn$, while the number of braidings over $\C(\Z_{m},P)$ is $m$ (see Appendix). Thus there are $2n$ different braidings over $\text{Fus}(\mathcal{MC})$.
    
    The second factor is equivalent as a fusion category to the category of $\Z_m$-graded vector spaces with trivial associativity constraint. The braidings on it are uniquely determined by $c_{1, 1}=\xi$, where $\xi$ is a $2^r$-th root of unity. Thus we have $2^r$ different braidings over the second factor. Notice that among all braidings, $b_{i, 2^{r-1}}=(\xi^{i2^{r-1}})^{2}=\xi^{i2^{r}}=1$, so the object labeled by $[2^{r-1}]$ is in the symmetric center.
    
    According to Corollary \ref{unique}, a braiding over $\Fus{(\C^{-})}$ is uniquely determined by braidings over the two factors. Therefore, the number of different braidings over $\Fus(\C^{-})$ is $2nm=2N$.
    
    \item [Case 2:] $k$ is even. \\
    We only show that the braiding over $\C^{-}$ is nondegenerate. Then the conclusion can be reached using the same argument in the proof of Theorem \ref{untwist}.
    
    We prove the non-degeneracy by checking the symmetric center. Let $Y$ be a simple object in the symmetric center of $\C^{-}$. Notice that the fusion ring of $\C^{-}$ and $\C$ are isomorphic. The corresponding object in $\C$ centralizes $\C$ projectively. In the following, we first show that $Y$ is pointed, then we examine all the pointed objects in $\C$ to show that $Y$ can only be the unit object.
    
    Since $Y$ centralizes $\C$ projectively, then we can define $b_{Y, Z}$ to be the quantity such that
    $$c_{Y, Z}\circ c_{Z, Y}=b_{Y, Z}\textbf{id}_{Z\otimes Y}$$
    for $Z$ simple. We also abuse this notation to non-simple $Z$ if $Y$ and $Z$ centralize each other. Notice \[b_{X, Y^{\otimes N}}=b_{X, Y}^{N}=b_{X^{\otimes N}, Y}=b_{\textbf{1}, Y}=1.\]
    Thus $Y^{\otimes N}$ centralizes $X$. Because $\C$ is nondegenerate, the only simple object in $\mathcal{C}$ that centralizes $X$ is the unit object. Therefore $Y^{\otimes N}$ is a direct sum of unit objects. Let $Z$ be a simple summand of $Y^{\otimes N-1}$. Then $Y\otimes Z$ is again a direct sum of unit objects. Moreover, $\C$ is rigid, so that $\Hom(Y\otimes Z, \textbf{1})=\Hom(Z, Y^{*})$. The latter space is the hom space of two simple objects, thus of dimension at most $1$. Therefore $Z$ must be $Y^{*}$ and $Y$ is pointed.
    
    Next we compute $b'_{X, g}$ in $\C^{-}$. Since $g$ is in grade $k$, $c'_{X, g}=s^{k}c_{X, g}$ and $c'_{g, X}=s^{k}c_{g, X}$. Then
    $$b'_{X, g}=s^{2k}b_{X, g}=t^{1+2kr}.$$
    Notice that $2kr$ divides all prime factors of $N$, thus $1+2kr$ is relatively prime to $N$, and $b'_{X, g}$ is still a primitive $N$-th root of unity. Therefore, none of the non-trivial  pointed objects are in the symmetric center, hence the braiding over $\C^{-}$ is nondegenerate.
\end{enumerate}
\end{proof}

\end{theorem}

\subsection{Braidings on Subcategories}
In this subsection, we classify all braided tensor categories $\D$ with the same fusion rule as the first factor $\mathcal{MC}$ in the decomposition formula \eqref{decomp}. We also show the braiding is unique up to certain symmetry.

\begin{theorem}\label{d}
Let $N=mn$ where $n=\text{gcd}(N, k^{\infty})$. Suppose $\mathcal{D}$ is a braided fusion category with the same fusion rules as
$$\mathcal{M}SU(N)_{k}=\bigoplus_{i=0}^{n-1}(SU(N) _k)^{im}.$$
Then 
\enumerate
\item The braiding over $\D$ must be nondegenerate.
\item The underlying monoidal category $\text{Fus}(\D)$ admits $2n$ different braidings.
\end{theorem}
\begin{proof}
Let $\P(\Z_{m})$ be a pointed fusion category with the same fusion rules as $\Vec_{\Z_{m}}$. Then, the fusion ring of $\D\boxtimes \P(\Z_{m})$ is isomorphic to the fusion ring of $SU(N)_k$. Thus $\D\boxtimes \P(\Z_{m})$ is monoidal equivalent to the monoidal category $\C(\sl_N, N+k, q))$ or the twisted monoidal category $\C(\sl_N,N+k, q)^{-}$. Both underlying fusion categories admit a nondegenerate braiding on the nonpointed factor in the decomposition formula \eqref{decomp}. On the other hand, we may choose a monoidal structure on the pointed factor $\P(\Z_{m})$, so that $\C:=\D\boxtimes \P(\Z_{m})$ admits a nondegenerate braiding. As a result, all braidings over the underlying fusion category of $\text{Fus}(\D)$ comes from the restriction of a braiding over $\text{Fus}(\C)$. Now we show for all the braidings over $\D$, the symmetric center must be trivial.

Since $\D$ admits a nondegenerate braiding, we only need to examine the pointed objects for categories equipped with other braidings. Due to the decomposition formula in \eqref{decomp}, all such objects can be written in the form $g^{jm}$, where $0\le j \le n-1$. Notice that $g^{jm}$ centralizes $\D$ if and only if $g^{jm}$ centralizes $X\in \mathcal{O}(\C)$. Let us choose a nondegenerate braiding $c_{-,-}$ over the underlying fusion category $\text{Fus}(\C)$. Then we have 
\begin{enumerate}
    \item the invariant $b_{X, g}=t$ is a primitive $N$-th root of unity;
    \item the braiding over $\C$ is in one to one correspondence with sections of $\C\boxtimes \C^{rev}\rightarrow \C$.
\end{enumerate}
Pick one braiding over $\C$ denoted by $\bar{c}$ that corresponding to the section $s$. We denote by $\bar{b}_{X, Y}$ the quantity such that
$$\bar{c}_{Y, X}\circ \bar{c}_{X, Y}=\bar{b}_{X, Y}\textbf{id}_{X\otimes Y}$$
if $X$ and $Y$ centralize each other.

Recall in \eqref{equat}, we have computed $\bar{b}_{X, g^{jm}}=t^{(2ki+1)jm}$. Thus, $g^{jm}$ is in the symmetric center if and only if 
$$(2ki+1)jm\equiv 0 \mod N,$$
which is equivalent to 
$$(2ki+1)j\equiv 0 \mod n.$$
Notice that $n$ is a divisor of $k^M$ for large $M$. This implies that $2ki+1$ is invertible in ring $\Z_{n}$ since
$$(2ki+1)(1-2ki+(2ki)^2-\cdots+(-1)^{M-1}(2ki)^{M-1})=1-(2ki)^M\equiv 1 \mod {n}.$$
Thus $g^{jm}$ is in the symmetric center if and only if $j=0$. This proves the first statement of Theorem \ref{d}.

The second statement can be proved by counting the number of braidings on $\C=\D\boxtimes \Z_{m}$ and the pointed factor $\mathcal{P}(\Z_{m})$. There are $2N$ different braidings over $\C$, and there are $m$ different braidings over $\Z_{m}$, see Appendix. Thus there are exactly $2N/m=2n$ distinct braidings over $\text{Fus}(\mathcal{D})$.
\end{proof}

\subsection{Autoequivalences and Braidings over $\Fus(\D)$.}
In this subsection, we consider how the group of autoequivalences of $\Fus(\D)$ acting on the set of braidings. Notice that all the autoequivalences of $\C$ for $\C$ Grothendieck equivalent to $SU(N)_k$ preserve the factorization $\C=\mathcal{MC}\boxtimes \C(\Z_{m}, P)$. Consequently, $\Eq(\C)=\Eq(\mathcal{MC})\times \Eq(\C(\Z_{m}, P))$. This allow us to study the group of autoequivalences of $\D$ as a subgroup of the group of autoequivalences of $\C$ as we constructed in the last subsection.

On the other hand, for the autoequivalence group of $\C$, we have the following exact sequence.
\[0\rightarrow\text{Gauge}(\C)\rightarrow \text{Eq}(\C)\xrightarrow{\text{Forget}} \text{Aut}(\mathcal{K}_{0}(\C))\]
where $\text{Gauge}(\C)$ is the subgroup of autoequivalences of $\C$ that fix the simple objects, and $\text{Aut}(\mathcal{K}_{0}(\C))$ is the group of automorphisms of the fusion ring of $\C$. By the following lemma, we show in our case, the Gauge equivalences are all braided.

\begin{lemma}\label{gauging}
Let $\mathcal{C}$ be a braided fusion category and $(F, J)\in \text{Gauge}(\C)$. Suppose $\C$ is tensor generated by a single object $X$ and $X\otimes X$ decomposes into distinct simple objects. Then $(F, J)$ is braided.
\end{lemma}

\begin{proof}
Since $F$ is an autoequivalence of $\C$, we denote the original braiding by $c$ and the braiding induced by $F$ by $c'$. Notice that the induced braiding $c'_{Y, Z}$ is induced by the following diagram.

$$\begin{tikzcd}
F(Y\otimes Z)\arrow[r, "F(c_{Y,Z})"] & F(Z\otimes Y)\arrow[d, "J_{Z, Y}^{-1}"] \\
F(Y)\otimes F(Z)\arrow[u, "J_{Y, Z}"]\arrow[r, "c'_{Y, Z}"] &  F(Z)\otimes F(Y)
\end{tikzcd}$$

In particular, $(F, J)$ is braided if the induced braiding $c'=c$. On the other hand, since our category is tensor generated by a single object $X$, in order to show $(F, J)$ is braided we only need to show $c_{X, X}=c'_{X, X}$.

Without loss of generality, we may assume $\C$ is skeletal. In this case, we have

$$\begin{tikzcd}
X\otimes X \arrow[r, "F(c_{X,X})"] & X\otimes X \arrow[d, "J_{X,X}^{-1}"] \\
X\otimes X\arrow[u, "J_{X, X}"]\arrow[r, "c'_{X, X}"] &  X\otimes X
\end{tikzcd}$$

Since $X\otimes X$ decomposes into distinct simple objects, $F(c_{X, X})=c_{X, X}$ and $\text{End}(X\otimes X)$ is commutative. Therefore, the composition 
\[c'_{X, X}:=J_{X, X}^{-1}c_{X, X}J_{X, X}=c_{X, X}.\]
This proves the lemma.
\end{proof}

With the lemma above, the group action of $\text{Eq}(\C)$ on the set of braidings descends to an action of the group of its image in $\text{Aut}(\mathcal{K}_{0}(\C))$. In particular, the simple current autoequivalences we constructed in Section 3 descend to a subgroup of $\text{Aut}(\mathcal{K}_{0}(\C))$. We denote this group by $\text{ScEq}(-)$, it acts on the set of braidings.

In the theorem below, we consider the subgroup $\text{ScEq}(\D)$, acting on the set of braidings over the underlying fusion category.
\begin{theorem}\label{smart}
Let $\D$ be a braided fusion category Grothendieck equivalent to $\mathcal{M}SU(N)_k$, and $N=mn$ where $n=\text{gcd}(N, k^{\infty})$. The underlying fusion category $\text{Fus}(\D)$ admits a unique braiding up to the following three actions:
\begin{enumerate}
    \item reverse the braiding;
    \item simple current autoequivalence;
    \item when $N$ and $k$ are both even, replace $c_{X, Y}$ by $-c_{X, Y}$ if both $X$ and $Y$ are of odd grade.
\end{enumerate}
The third action is needed only in the case that $N\equiv 2 \mod 4$ and $k\equiv 2\mod 4$.
\end{theorem}

We remark that the first and third operations always induce different braidings. On the other hand, simple current autoequivalences could be braided: there are exactly $2$ braided simple current autoequivalence only if $N\equiv 2 \mod 4$ and $k\equiv 2\mod 4$.

\begin{proof}
The order of $\text{ScEq}(\D)$ is $n$, by an exactly the same computation as in \cite{ediesimple} on checking the braiding morphism on $X\otimes X$, the subgroup of braided autoequivalences is 
\begin{enumerate}
    \item[(i)] $\Z_{2}$ if $N\equiv 2 \mod 4$ and $k\equiv 2\mod 4$;
    \item[(ii)] trivial otherwise.
\end{enumerate}
Since simple current autoequivalences do not change the ratio of the eigenvalues of the braid morphism, $\text{ScEq}(\D)$ acts on the set of $n$ braidings corresponding to those sections in $\D\boxtimes \D^{rev}$ with only pointed objects on the reversed side.

The case (ii) follows directly by the orbit-stablizer theorem. For case (i), we need to show the action (3) cannot be realized by simple current autoequivalences. Let $\C$ be a nondegenerate braided fusion category such that $\mathcal{K}_{0}(\C)=\mathcal{K}_{0}(SU(N)_k)$ with $\mathcal{MC}=\D$. Notice that the action (3) on $\D$ induces the same action on $\C$. Suppose $F$ is a simple current autoequivalence of $\C$ that sends $X$ to $F(X)=X\otimes g^{a}$ which induces the same braiding as action $(3)$. Then the induced braiding
$$\tilde{c}_{X, X}=b_{X, g^{a}}c_{g^a, g^a}c_{X, X}=-c_{X, X}.$$
Observe that $b_{X, g}=t$ is a primitive $N$-th root of unity. The above equation is simplified to
$$t^{a+\frac{ka^2}{2}}=-1,$$
or equivalently
$$a+\frac{ka^2}{2}\equiv \frac{N}{2} \mod N.$$
The above equation does not have a solution if $N\equiv 2 \mod 4$ and $k\equiv 2\mod 4$. This completes the proof. 
\end{proof}

\noindent
{\bf Acknowledgements:}
E.C.R. and Z.F. were partially supported by NSF grant DMS-1664359. E.C.R. was also partially supported by NSF grant DMS-2000331.

\appendix 
\section{Braidings over $\Z_N$} In this appendix, we compute all possible twists of $\Z_N$ that admit a braid structure, it is a special case of \cite[Exercise 8.4.3]{EGNO}.
\begin{proposition}\label{twisting}
The category $\Vec_{\Z_N}^{\omega}$ has a braid structure if and only if
\begin{enumerate}
    \item $\omega$ is trivial if $N$ is odd;
    \item $\omega^2$ is trivial if $N$ is even.
\end{enumerate}
In particular, $\omega$ is trivial if and only if the following conditions hold: For all objects $l$ with order $2^p$, the quantity $b_{l, l}$ has order less than $2^p$.
\end{proposition}
\begin{proof}
Without loss of generality, we assume the category is skeletal with simple objects $\{0, 1, 2,\ldots, N-1\}$. We take the following representative $3$-cocycles
$$\omega(i, j, k)=\begin{cases} 1 & i+j<N\\
\eta^{k} & i+j\geq N \end{cases}$$
where $\eta$ is an $N$-th root of unity.

The category admits a braid structure if $c_{i, j}:=c(i, j)\textbf{id}_{i\otimes j}$ satisfies the (hexagon) equations
\begin{equation}
    \omega(j,k,i)c(i, j+k)\omega(i,j,k)=c(i,k)\omega(j,i,k) c(i,j)\end{equation}
\begin{equation}\omega(k,i,j)^{-1}c(i+j, k)\omega(i,j,k)^{-1}=c(i,k)\omega(i,k,j)^{-1}c(j, k).\end{equation}

With our choice of the $3$-cocycles, $\omega(i,j,k)=\omega(j,i,k)$ so the equations simplify to
\begin{equation}\label{3}
    \omega(j,k,i)c(i, j+k)=c(i,k) c(i,j)\end{equation}
\begin{equation}\label{4}
    c(i+j, k)\omega(i,j,k)^{-1}=c(i,k)c(j, k).\end{equation}
Suppose $c(1, 1)=s$, apply the equation \ref{3} inductively, one get $c(i, j)=s^{ij}$ for $0<i<N$ and $0<j<N$.

Choose $j, k$ such that $j+k=N$, equation \ref{3} becomes
$$\eta^i=s^{ik}s^{ij}=s^{iN}.$$

Choose $i, j$ such that $i+j=N$, equation \ref{4} becomes
$$\eta^{-k}=s^{ik}s^{jk}=s^{kN}.$$

Thus $\eta^{2i}=1$ for all $i$, $\eta=1$ or $-1$. Since $\eta$ is an $N$-th root of unity. $\eta=-1$ only if $N$ is even. This proves the first part of the proposition.

In order to prove the second part, we may assume $N=2^q$, the category factors otherwise. Notice $\eta=s^{2^q}$ and $b_{l, l}=s^{2l^2}$. $\eta=1$ if and only if $s$ has order less than $2^{q}$. Since order of $s$ must be power of $2$, the later condition is equivalent to $s^{2l^2}$ has order less than $2^{N-(2N-2p+1)}=2^{2p-1-N}$. The ``only if'' part is obvious and the ``if'' part comes from plugging in $p=N$.
\end{proof}

We get the following corollary immediately:
\begin{corollary}
Suppose $\mathcal{P}$ is a pointed fusion category that admits a braiding. Then its full subcategory $\mathcal{P}^{2}$ consisting of objects of even grades must have trivial associativity constraints.
\end{corollary}

\begin{corollary}
Let $\mathcal{P}$ be a braided fusion category Grothendieck equivalent to $\Vec_{\Z_N}$. Then $\mathcal{P}$ admit $N$ braid structures.
\end{corollary}

\begin{proof} As a monoidal category, $\mathcal{P}$ is monoidal equivalent to $\Vec_{\Z_N}^\omega$ for some 3-cocycle $\omega$ by \cite[Proposition 2.6.1(iii)]{EGNO}.
As in the proof of the proposition, we let $\1$ be the generating object, and $c(1, 1)=s$. We compute the possible braidings case by case. Suppose $N=2^r q$ where $q$ is an odd number.
\begin{enumerate}
    \item [Case 1] $r=0$.
Since $N$ is odd, we have $s^N=1$, there are exactly $N$ different choices of $s$.
\item [Case 2] $r\ge 0$ with trivial $\omega$.
By the proof of the proposition we have $s^N=1$ thus there are exactly $N$ different choices of $s$.
\item [Case 3] $r\ge 0$ with nontrivial $\omega$.
By the proof of the proposition we have $s^N=-1$ so $s$ is a $2N$-th roots of unity, among all the $2N$-th root of unity, half of them satisfy $s^N=-1$.
\end{enumerate}
\end{proof}




\bibliographystyle{plain}
\bibliography{reference}

\end{document}